\documentclass[runningheads]{llncs}
\usepackage[T1]{fontenc}
\usepackage{graphicx}

\usepackage{amsmath}
\allowdisplaybreaks
\usepackage{amsfonts}
\usepackage{stmaryrd}
\usepackage{multicol, multirow}
\usepackage[svgnames]{xcolor}
\usepackage{mathtools}
\usepackage{nicefrac}
\usepackage{bm}
\usepackage{booktabs}
\usepackage{hyperref}
\usepackage{cite}

\usepackage[draft, todonotes={textsize=tiny}]{changes}
\setuptodonotes{color=red, backgroundcolor=red!60!white,
  bordercolor=red, tickmarkheight=10pt}
  
\definechangesauthor[name=Robert, color=cyan]{RB}

\definechangesauthor[name=Jonas, color=magenta]{JB}

\spnewtheorem{algorithm}{Algorithm}{\bfseries}{}

\DeclareMathOperator*{\argmin}{arg\,min}

\DeclareMathOperator*{\prox}{prox}

\DeclareMathOperator*{\TV}{TV}

\DeclareMathOperator*{\dist}{dist}

\newcommand{\R}{\mathbb R}

\newcommand{\Fun}[1]{\mathcal{#1}}
\newcommand{\FF}{\Fun F}
\newcommand{\FH}{\Fun H}
\newcommand{\FI}{\Fun I}
\newcommand{\FJ}{\Fun J}
\newcommand{\FK}{\Fun K}

\newcommand{\FQ}{\Fun Q}

\newcommand{\FG}{\Fun G}
\newcommand{\FL}{\Fun L}

\newcommand{\FV}{\Fun V}

\newcommand{\Mat}[1]{\bm{#1}}
\newcommand{\MA}{\Mat A}

\newcommand{\ME}{\Mat E}
\newcommand{\MI}{\Mat I}

\newcommand{\MSigma}{\Mat \Sigma}

\newcommand{\MQ}{\Mat Q}

\newcommand{\MU}{\Mat U}
\newcommand{\MV}{\Mat V}

\newcommand{\MY}{\Mat Y}
\newcommand{\MZ}{\Mat Z}
\newcommand{\Mzero}{\Mat 0}

\newcommand{\Vek}[1]{\bm{#1}}

\newcommand{\Ve}{\Vek e}
\newcommand{\Vf}{\Vek f}

\newcommand{\Vx}{\Vek x}
\newcommand{\Vy}{\Vek y}
\newcommand{\Vz}{\Vek z}

\newcommand{\Vu}{\Vek u}

\newcommand{\Vv}{\Vek v}

\newcommand{\Vxi}{\Vek \xi}
\newcommand{\Vell}{\Vek \ell}

\newcommand{\Vones}{\Vek 1}

\newcommand{\Msym}{\mathcal S}

\newcommand{\hyper}{\mathbb H}
\newcommand{\Pdisc}{\mathbb D}
\newcommand{\Phalf}{\mathbb P}
\newcommand{\Aff}{\mathbb A}

\newcommand{\proj}{{\mathrm{proj}}}
\newcommand{\arcosh}{{\mathrm{arcosh}}}
\newcommand{\diag}{{\mathrm{diag}}}

\newcommand{\tF}{\mathrm{F}}

\newcommand{\rk}{\mathrm{rank}}

\newcommand{\Mbiform}{\eta}
\newcommand{\Mdist}{\dist_{\hyper_d}}

\begin{document}
\title{Denoising Hyperbolic-Valued Data by
Relaxed Regularizations%
}

\author{R. Beinert \and J. Bresch}
\authorrunning{R. Beinert and J. Bresch}
\institute{Technische Universität Berlin, 
Institute of Mathematics, Stra{\ss}e des 17. Juni 136, 10623 Berlin, Germany
\email{\{beinert, bresch\}@math.tu-berlin.de}}
\maketitle             
\begin{abstract}
    We introduce a novel relaxation strategy
    for denoising hy\-per\-bol\-ic-valued data.
    The main challenge is here
    the non-convexity of the hyperbolic sheet.
    Instead of considering the denoising problem directly 
    on the hyperbolic space,
    we exploit the Euclidean embedding
    and encode the hyperbolic sheet
    using a novel matrix representation.
    For denoising,
    we employ the Euclidean Tikhonov 
    and total variation (TV) model,
    where we incorporate our matrix representation.
    The major contribution is then
    a convex relaxation of the variational ansätze
    allowing the utilization 
    of well-established convex optimization procedures 
    like the alternating directions method of multipliers (ADMM).
    The resulting denoisers 
    are applied to a real-world Gaussian image processing task,
    where we simultaneously restore the pixelwise
    mean and standard deviation 
    of a retina scan series.
\end{abstract}

% -------------------------------------------------------------------------------------------
\section{Introduction}
% -------------------------------------------------------------------------------------------

The geometry of hyperbolic spaces finds
applications across diverse research fields
spanning from theoretical studies in mathematics and physics
to practical implementations in machine learning, 
network analysis, 
and computer vision.
One of the most famous example is the specific relativity theory,
where hyperbolic spaces and their geometry are often used to model the spacetime
\cite{Wa84}.
Moreover,
hyperbolic spaces are particularly well-suited
for capturing hierarchical structures, 
such as those found in social and biological networks
as well as in the topology of the internet 
\cite{KPKVB10,BPK10}.
In order to analyse hierarchical or tree-like data,
the hyperbolic geometry has been incorporated
into machine learning architectures
leading to so-called hyperbolic neural networks
\cite{LiuNicKie2019,GanBecHof2018},
which,
for instance,
play a major role in hierarchical clustering
\cite{NicKie2017,LinBai2023}.
From the view point of image analysis,
hyperbolic spaces become interesting
due to their utilization in
image segmentation \cite{ASAV+22},
image denoising \cite{BerPerSte16},
and morphological image processing \cite{AV14}.
Hyperbolic-valued images, 
for instance, 
occur in probabilistic image analysis \cite{AV14,BerPerSte16},
where each pixel is described by a Gaussian distribution,
and as structure tensor in edge and texture processing \cite{CF09}.

Mathematically,
the \emph{hyperbolic space} $\hyper_d$ is the 
unique, 
simply connected,
$d$-dimensional
Riemannian manifold 
with constant curvature of $-1$,
see \cite{Killing1891,Hopf1926}.
The hyperbolic spaces possess several different embeddings
\cite{L97}.
Throughout this paper,
we rely on the Euclidean one:
\begin{equation*}
    \hyper_d \coloneqq \{\Vx \in \R^{d+1} \colon \Mbiform(\Vx, \Vx) = -1 \quad\text{and}\quad x_{d+1} > 0\}
\end{equation*}
where $\Mbiform \colon \hyper_d \times \hyper_d \to \R, (\Vx, \Vy) \mapsto \sum_{i = 1}^d x_iy_i - x_{d+1}y_{d+1}$
is the \emph{Minkowski bilinearfrom} 
defining the \emph{hyperbolic distance}
$\Mdist(\Vx, \Vy) \coloneqq \arcosh(-\Mbiform(\Vx, \Vy)).$
Note that 
the Euclidean embedding of the hyperbolic space $\hyper_d$,
which is sometimes called the \emph{hyperbolic sheet},
is contained in the $(d+1)$-dimensional, affine halfspace 
$\Aff_{d+1} \coloneqq \R^d \times \R_{\geq 1}$,
i.e.\ $\hyper_d \subset \Aff_{d+1}$.

In this paper,
we approach the problem of hyperbolic image denoising.
Here,
the main issue comes form the non-convexity of the hyperbolic sheet.
This is why hyperbolic image denoising becomes highly demanding
from a numerical point of view.
For instance,
hyperbolic images can be denoised
using total variation (TV) models based on the hyperbolic distance
and applying a manifold version of the Douglas--Rachford algorithm
\cite{BerPerSte16}.
The major contribution of this paper is that,
instead of minimizing a variational denoising model
directly on the hyperbolic sheet $\hyper_d$,
we exploit the Euclidean embedding
and propose a novel convex relaxation
that allows the application of well-established convex optimization methods on $\R^{d+1}$.
The paper is organized as follows:
In §~\ref{sec:DenoisingHd},
we introduce (Euclidean) TV- and Tikhonov-like regularization models
and convex relaxations
that encode the hyperbolic sheet 
using positive semi-definite matrices
in the spirit of \cite{BeBrSt24,BeBr24,condat_1D2D}.
In §~\ref{sec:algo},
we use the alternating directions method of multipliers (ADMM) \cite{GM75}
to derive an actual denoising algorithm.
Finally,
in §~\ref{sec:num-exp},
we apply the proposed method to 
denoising problems from Gaussian image processing.

%---------------------------------------------------------------------------------------------
\section{Denoising of Hyperbolic-Valued Data} \label{sec:DenoisingHd}
%---------------------------------------------------------------------------------------------

The aim of this paper is to denoise hyperbolic-valued data
on a connected, undirected graph $G = (V, E)$,
where $V \coloneqq \{1, \dots, N\}$ denotes the set of $N$ vertices
and $E \subseteq \{(n,m) \in V \times V: n < m \}$
the set of $M \coloneqq |E|$ edges.
We are now interested in the restoration of hyperbolic-valued data 
$\Vx \coloneqq (\Vx_n)_{n \in V} \in \hyper_d^{N}$
form noisy measurements $\Vy \coloneqq (\Vy_n)_{n \in V}$ 
in $\hyper_d^N$ or $(\R^{d+1})^N$.
For applications in signal and image processing,
$G$ corresponds to the line or grid graph.

%---------------------------------------------------------------------------------------------
\subsection{Tikhonov-like Regularization}   \label{sec:Tik_denoising_model}
%---------------------------------------------------------------------------------------------

For denoising relatively smooth hyperbolic data,
we consider the variational Tikhonov-like denoising model:
\begin{equation}
\label{eq:orig_tik_Hd}
    \argmin_{\Vx \in \hyper_d^N}
    \quad
    \frac{1}{2}\sum_{n \in V} \lVert \Vx_n - \Vy_n\rVert_2^2
    + 
    \frac{\lambda}{2}\sum_{(n,m) \in E} 
    \lVert\Vx_n - \Vx_m\rVert_2^2,
\end{equation}
where $\lVert \cdot \rVert_2$ denotes the Euclidean norm,
and $\lambda > 0$ the regularization parameter.
Note that 
the model \eqref{eq:orig_tik_Hd} is non-convex
due to the minimization of the hyperbolic sheet.
In order to convexify \eqref{eq:orig_tik_Hd},
we adapt the strategies in \cite{BeBrSt24, condat_1D2D},
where the Tikhonov denoising for the hypersphere and SO(3) is considered.
For this,
we introduce the auxiliary variables $\Vf \in \R^M$ and $\Vv \in \R^N$ 
and the linear objective 
\begin{equation}
    \label{eq:func_J}
    \FJ(\Vx, \Vf, \Vv) 
    \coloneqq 
    \frac{1}{2}\smashoperator[lr]{\sum_{n \in V}}
    \left(\Vv_n
    - 2 \langle \Vx_n, \Vy_n\rangle
    \right)
    + \frac{\lambda}{2}
    \smashoperator[lr]{\sum_{(n,m) \in E}}
    \left(\Vv_n + \Vv_m
    - 2\Vf_{(n,m)}\right)
\end{equation}
with the Euclidean inner product $\langle \cdot\, , \cdot \rangle$.
Therewith,
\eqref{eq:orig_tik_Hd} becomes equivalent to
\begin{equation}
    \argmin_{\Vx \in \hyper_d^N}
    \quad
    \FJ(\Vx, \Vf, \Vv)
    \quad\text{s.t.}\quad
    \left\{\begin{array}{ll}
        \Vv_n = \lVert\Vx_n\rVert_2^2 & \forall n \in V\\
        \Vf_{(n,m)} = \langle \Vx_n, \Vx_m\rangle \quad & \forall (n,m) \in E
    \end{array}\right\}.
    \label{eq:rew_tik_Hd}
\end{equation}
For any $\Vxi \in \R^{d+1}$,
we define $\tilde\Vxi \coloneqq (\xi_1, ..., \xi_d, -\xi_{d+1})^*$.
Further,
we denote 
the $(d+1)\times(d+1)$ identity
by $\MI_{d+1}$.
To encode the minimization over the hyperbolic sheet,
we introduce%
---based on the additional variables $\Vell \in \R^M$---%
the series of matrices
\begin{equation*}
    \MQ_{(n,m)} 
    \coloneqq
    \scalebox{0.8}{$\begin{aligned}\left[\begin{array}{c | c c | c c}
        \MI_{d+1} & \Vx_{n} & \tilde\Vx_{n} & \Vx_{m} & \tilde\Vx_{m} \\
        \hline
        \Vx_{n}^* & \Vv_n & -1 & \Vf_{(n,m)} & \Vell_{(n,m)}\\
        \tilde\Vx_{n}^* & -1 & \Vv_n & \Vell_{(n,m)} & \Vf_{(n,m)}\\
        \hline
        \Vx_{m}^* & \Vf_{(n,m)} & \Vell_{(n,m)} & \Vv_m & -1\\
        \tilde\Vx_{m}^* & \Vell_{(n,m)} & \Vf_{(n,m)} & -1 & \Vv_m
    \end{array}\right]\end{aligned}$}.
\end{equation*}

\begin{proposition}
    \label{prop:SmallMatPres}
    Let $\Vx_n, \Vx_m \in \Aff_{d+1}$.
    Then $\Vx_n, \Vx_m \in \hyper_d$, 
    $\Vv_n = \lVert\Vx_n\rVert_2^2$,
    $\Vv_m = \lVert\Vx_m\rVert_2^2$,
    $\Vell_{(n,m)} = \Mbiform(\Vx_n, \Vx_m)$,
    and $\Vf_{(n,m)} = \langle \Vx_n, \Vx_m\rangle$
    if and only if $\MQ_{(n,m)} \succeq \Mzero$
    and $\rk(\MQ_{(n,m)}) = d+1$.
\end{proposition}

\begin{proof}
    Because of the identity in the upper left,
    $\MQ_{(n,m)}$ has at least rank $d+1$.
    Now, 
    the last four rows can be represented by the first $d+1$ rows
    if and only if
    $\Vv_n = \lVert\Vx_n\rVert_2^2$,
    $\Vv_m = \lVert\Vx_m\rVert_2^2$,
    $\Vell_{(n,m)} = \Mbiform(\Vx_n, \Vx_m)$,
    $\Vf_{(n,m)} = \langle \Vx_n, \Vx_m\rangle$
    as well as
    $\Mbiform(\Vx_n,\Vx_n) = -1$
    and $\Mbiform(\Vx_m,\Vx_m) = -1$,
    i.e.\ $\Vx_n, \Vx_m \in \hyper_d$.
    Applying Schur's theorem \cite{HJ13}, 
    we have $\MQ_{(n,m)} \succeq \Mzero$
    if and only if
    the Schur complement
    \begin{equation*}
        \MQ_{(n,m)} / \MI_{d+1}
        =
        \scalebox{0.8}{$\begin{aligned}\left[\begin{array}{c c | c c}
             \Vv_n - \lVert\Vx_n\rVert_2^2 & - 1 - \Mbiform(\Vx_n, \Vx_n)  & \Vf_{(n,m)} - \langle \Vx_n, \Vx_m\rangle & \Vell_{(n,m)} - \Mbiform(\Vx_n, \Vx_m)\\
            - 1 - \Mbiform(\Vx_n, \Vx_n) & \Vv_n - \lVert\Vx_n\rVert_2^2 & \Vell_{(n,m)} - \Mbiform(\Vx_n, \Vx_m) & \Vf_{(n,m)} - \langle \Vx_n, \Vx_m\rangle \\
            \hline
            \Vf_{(n,m)} - \langle \Vx_n, \Vx_m\rangle & \Vell_{(n,m)} - \Mbiform(\Vx_n, \Vx_m) & \Vv_m - \lVert\Vx_m\rVert_2^2& - 1 - \Mbiform(\Vx_m, \Vx_m)\\
            \Vell_{(n,m)} - \Mbiform(\Vx_n, \Vx_m) & \Vf_{(n,m)} - \langle \Vx_n, \Vx_m\rangle & - 1 - \Mbiform(\Vx_m, \Vx_m) & \Vv_m - \lVert\Vx_m\rVert_2^2 
        \end{array}\right]\end{aligned}$}
    \end{equation*}
    is positive semi-definite. 
    Under the given assumptions,
    $\MQ_{(n,m)} / \MI_{d+1}$ becomes zero
    and hence positive semi-definite.
    \hfill$\Box$
\end{proof}

In light of Prop.~\ref{prop:SmallMatPres},
the Tikhonov-like denoising model \eqref{eq:rew_tik_Hd} is equivalent to
\begin{equation}       
    \argmin_{\substack{\Vx \in \Aff_{d+1}^N,\,  \Vell \in \R^M,\\ \Vf \in \R^M,\, \Vv \in \R^N}}
    \FJ(\Vx, \Vf, \Vv)
    \quad\text{s.t.}\quad
    \left\{
    \begin{array}{c}
        \MQ_{(n,m)} \succeq \Mzero \\
        \rk(\MQ_{(n,m)}) = d+1
    \end{array}
    \right\}
     \; \forall (n,m) \in E.
     \label{prob:UnConvRelaxedHyper}
\end{equation}
In the manner of \cite{BeBrSt24},
we convexify \eqref{prob:UnConvRelaxedHyper}
by neglecting the non-convex rank-constraints.
Thus,
we propose to solve the convex \emph{relaxed Tikhonov denoising model}:
\vspace{-0.2cm}
\begin{equation}   
\label{eq:conv_tik_Hd}
    \argmin_{\substack{\Vx \in \Aff_{d+1}^N, \, \Vell \in \R^M, \\ \Vf \in \R^M, \, \Vv \in \R^N}}
    \FJ (\Vx, \Vf, \Vv) 
    \quad \text{s.t.} \quad 
    \MQ_{(n,m)} \succeq \Mzero
    \quad \forall (n,m) \in E.
\end{equation}

%---------------------------------------------------------------------------------
\subsection{TV Regularization}
%---------------------------------------------------------------------------------

For denoising cartoon-like data,
the Tikhonov denoising model in §~\ref{sec:Tik_denoising_model} is not suitable.
Instead,
we consider the variational TV denoising model:
\begin{equation}
    \label{eq:orig_tik_Hd_TV}
    \argmin_{\Vx \in \hyper_d^N}
    \frac{1}{2}\sum_{n \in V}\lVert\Vx_n - \Vy_n\rVert_2^2
    + \mu \TV(\Vx),
    \quad\text{where}\quad
    \TV(\Vx) 
    \coloneqq 
    \smashoperator{\sum_{(n,m) \in E}}\lVert\Vx_n - \Vx_m\rVert_1
\end{equation}
and $\lVert\cdot\rVert_1$ denotes the Manhattan norm,
and $\mu > 0$ is the regularization parameter.
In the manner of \cite{BeBr24},
where the TV denoising on the hypersphere and SO(3) is considered,
we introduce the auxiliary parameter $\Vv \in \R^N$
and the (non-linear) objective 
\begin{equation*}
    \FK (\Vx, \Vv) 
    \coloneqq 
    \frac{1}{2}\sum_{n \in V} (\Vv_n - 2\langle \Vx_n,\Vy_n\rangle) 
    + \mu \TV(\Vx).
\end{equation*}
Hence,
\eqref{eq:orig_tik_Hd_TV} becomes equivalent to 
\begin{align}
    \label{eq:rew_tik_Hd_TV}
    \argmin_{\Vx \in \hyper_d^N}
    \;
    \FK(\Vx, \Vv)
    \quad\text{s.t.}\quad
    \Vv_n = \lVert\Vx_n\rVert_2^2
    \quad\forall n \in V.
\end{align}
Applying Prop.~\ref{prop:SmallMatPres},
we may rewrite \eqref{eq:rew_tik_Hd_TV} as
\begin{align}
    \argmin_{\substack{\Vx \in \Aff_{d+1}^N, \Vell \in \R^M \\
    \Vf \in \R^M,\Vv \in \R^N}}
    \FK (\Vx, \Vv) 
    \quad\text{s.t.}\quad
    \left\{\begin{array}{c}
        \MQ_{(n,m)} \succeq 0 \\
        \rk(\MQ_{(n,m)}) = d+1
    \end{array}\right\}
    \quad \forall
    (n,m) \in E.
    \label{eq:orig_tik_Hd_TV_rew}
\end{align}
Different from §~\ref{sec:Tik_denoising_model},
where $\Vf$ is essential to linearize the objective,
and where $\MQ_{(n,m)}$ encodes $\Vf_{(n,m)} = \langle \Vx_n, \Vx_m \rangle$,
the auxiliary variables $\Vf$ and $\Vell$ here seem to be superfluous.
We derive an alternative encoding of the hyperbolic sheet.

\begin{proposition}
    \label{prop:TinyMatPres}
    Let $\Vx_n \in \Aff_{d+1}$.
    Then $\Vx_n\in \hyper_d$ and $\Vv_n = \lVert\Vx_n\rVert_2^2$
    if and only if 
    \begin{equation*}
        \MV_n 
        \coloneqq 
        \scalebox{0.8}{$\begin{aligned}\left[\begin{array}{c | c c}
            \MI_{d+1} & \Vx_{n} & \tilde\Vx_{n} \\
            \hline
            \Vx_{n}^* & \Vv_n & -1 \\
            \tilde\Vx_{n}^* & -1 & \Vv_n
        \end{array}\right]\end{aligned}$}
        \succeq 0
        \quad\text{and}\quad
        \rk(\MV_n) = d+1.
    \end{equation*}
\end{proposition}

The statement can be established in analogy to Prop.~\ref{prop:SmallMatPres}
using Schur's theorem.
On the basis of Prop.~\ref{prop:TinyMatPres},
the TV denoising model \eqref{eq:rew_tik_Hd_TV} is also equivalent to 
\vspace{-0.1cm}
\begin{equation*}
    \argmin_{\substack{\Vx \in \Aff_{d+1}^N, \Vv \in \R^N}} 
    \FK(\Vx, \Vv)
    \quad\text{s.t.}\quad
    \left\{\begin{array}{c}
        \MV_{(n)} \succeq 0 \\
        \rk(\MV_{(n)}) = d+1
    \end{array}\right\}
    \quad\text{for all}\quad
    n \in V.
    %\label{eq:red_tik_Hd_TV}
\end{equation*}
Neglecting the rank constraint,
we propose to solve the convex \emph{relaxed TV denoising model}:
\begin{equation}   
    \label{eq:conv_red_tik_Hd_TV}
    \argmin_{\substack{\Vx \in \Aff_{d+1}^N, \Vv \in \R^N}}
    \FK (\Vx, \Vv) 
    \quad \text{s.t.} \quad 
    \MV_{n} \succeq 0
    \quad \forall n \in V.
\end{equation}
Notice
that the relaxed version of \eqref{eq:orig_tik_Hd_TV_rew}%
---without the rank constraint---%
and the proposed model \eqref{eq:conv_red_tik_Hd_TV} are equivalent:
if $(\hat\Vx, \hat\Vell, \hat\Vf, \hat\Vv)$ is a solution of the relaxed version of \eqref{eq:orig_tik_Hd_TV_rew},
then $(\hat\Vx, \hat\Vv)$ is a solution of \eqref{eq:conv_red_tik_Hd_TV};
if $(\hat\Vx, \hat\Vv)$ is a solution of \eqref{eq:conv_red_tik_Hd_TV},
then $(\hat\Vx, (\eta(\hat\Vx_n, \hat\Vx_m))_{(n,m) \in E}, (\langle \hat\Vx_n, \hat\Vx_m\rangle)_{(n,m) \in E}, \hat\Vv)$
is a solution of the relaxed version of \eqref{eq:orig_tik_Hd_TV_rew}.
The positive semi-definiteness can be transferred between $\MQ_{(n,m)}$, $\MV_n$, and $\MV_m$
utilizing Schur's theorem and Sylvester's criterion \cite{HJ13}.

% -------------------------------------------------------------------------------------
\section{Denoising Models and Algorithms}          \label{sec:algo}
% -------------------------------------------------------------------------------------

In this section
we solve both convex relaxed problems,
defined in \eqref{eq:conv_tik_Hd} and \eqref{eq:conv_red_tik_Hd_TV},
employing the alternating directions methods of multipliers (ADMM) \cite{GM75}.
In general,
ADMM allows to solve convex minimization problems of the form
\begin{align}
    \label{eq:split}
    \argmin\nolimits_{\MY, \MU}
    \quad
    \FF(\MY) + \FG(\MU)
    \quad\text{s.t.}\quad
    \FL(\MY) = \MU,
\end{align}
where $\FF$ and $\FG$ are convex,
and $\FL$ is linear,
using the iteration:
\begin{subequations}
\label{eq:ADMM}
\begin{align}
    \MY^{(k+1)}
    & \coloneqq
    \argmin\nolimits_{\MY} 
    \quad 
    \FF(\MY)
    + \tfrac{\rho}{2} \lVert\FL(\MY) - \MU^{(k)} + \MZ^{(k)}\rVert^2,
    \label{eq:ADMM_I}\\
    \MU^{(k+1)}
    & \coloneqq
    \argmin\nolimits_{\MU} 
    \quad
    \FG(\MU)
    + \tfrac{\rho}{2} \lVert\FL(\MY^{(k+1)}) - \MU + \MZ^{(k)}\rVert^2, 
    \label{eq:ADMM_II}\\
    \MZ^{(k+1)}
    & \coloneqq
    \MZ^{(k)} + \FL(\MY^{(k+1)}) - \MU^{(k+1)}.
    \label{eq:ADMM_III}
\end{align}
\end{subequations}

% -----------------------------------------------------------------------------------------------
\subsection{Solving the Relaxed Tikhonov-like denoising model}      \label{sec:sol_relaxed_tik}
% -----------------------------------------------------------------------------------------------

To solve \eqref{eq:conv_tik_Hd},
we rely on the following splitting:
\begin{align}
    \label{eq:ADMM_Tik}
    \argmin_{\Vx, \Vell, \Vf, \Vv, \MU}
        \FJ(\Vx, \Vf, \Vv)
        +
        \FI(\Vx)
        + 
        \FG(\MU)
        \quad\text{s.t.}\quad
        \FQ(\Vx, \Vell, \Vf, \Vv) = \MU
        \in \Msym_{d+5}^M,
\end{align}
where $\Msym_{d+5}$ denotes the symmetric $(d+5) \times (d+5)$ matrices. 
Using the indicator function $\iota_\bullet$,
which is zero on the indicated set
and $+\infty$ otherwise,
we define 
\begin{align}
    \label{eq:indicator_func}
    \FI(\Vx) 
    \coloneqq 
    \sum_{n \in V} \iota_{\Aff_{d+1}}(\Vx_n)
    \quad\text{and}\quad
    \FG(\MU) 
    \coloneqq 
    \sum_{(n,m) \in E} \iota_{\mathcal C}(\MU_{(n,m)}) \\[-20pt]\notag
\end{align}
with
$\ME \coloneqq \diag(\MI_{d+1}, 
    \left[\begin{smallmatrix}
        0 & -1 \\ -1 & 0
    \end{smallmatrix}\right],
    \left[\begin{smallmatrix}
        0 & -1 \\ -1 & 0
    \end{smallmatrix}\right])$ 
and $\mathcal C \coloneqq \{\MA \in \Msym_{d+5}: \MA \succeq - \ME\}$.
Finally,
the linear operator $\FQ \coloneqq (\FQ_{(n,m)})_{(n,m) \in E}$
is defined as
\begin{align*}
    \FQ_{(n,m)}(\Vx, \Vell, \Vf, \Vv) \coloneqq \MQ_{(n,m)} - \ME.
\end{align*}

To apply \eqref{eq:ADMM},
we define $\lVert\MU\rVert^2 \coloneqq \sum_{(n,m) \in E} \lVert\MU_{(n,m)}\rVert_{\tF}^2$,
where $\lVert\cdot\rVert_{\tF}$ denotes the Frobenius norm.
Further,
we require the adjoint of $\FQ$.
Based on the Kronecker delta $\delta_{i,j}$,
the adjoint of $\FQ$ with respect to the argument $\Vx$ is given by
\begin{align*}  
    \bigl((\FQ_{\Vx}^* (\MU))_{n}\bigr)_i
    &= 2 \, \Bigl[
    \smashoperator[r]{\sum_{(n,m) \in E}} \,
    (\MU_{(n,m)})_{i,d+2} + (-1)^{\delta_{i,d+1}}(\MU_{(n,m)})_{i,d+3} 
    \Bigr] \notag \\[-10pt]
    & \hspace{2cm} 
    + 2 \, \Bigl[
    \smashoperator[r]{\sum_{(m,n) \in E}} \,
    (\MU_{(m,n)})_{i,d+4} + (-1)^{\delta_{i,d+1}}(\MU_{(m,n)})_{i,d+5} 
    \Bigr]
\end{align*}
for $n \in V$ and $1 \leq i \leq d+1$;
the adjoint with respect to $\Vell$ and $\Vf$ are given by
\begin{align*}        
    (\FQ_{\Vell}^* (\MU))_{(n,m)}
    = 2 \left[(\MU_{(n,m)})_{d+5,d+2} + (\MU_{(n,m)})_{d+4,d+3}\right]\\[-12pt]
    \\
    (\FQ_{\Vf}^* (\MU))_{(n,m)}
    = 2\left[(\MU_{(n,m)})_{d+4,d+2} + (\MU_{(n,m)})_{d+5,d+3}\right]
\end{align*}   
for $(n,m) \in E$;
and the adjoint with respect to $\Vv$ is given by
\begin{align*}           
    (\FQ_{\Vv}^* (\MU))_{n}
    & = \Bigl[
    \smashoperator[r]{\sum_{(n,m) \in E}} \,
    (\MU_{(n,m)})_{d+2,d+2} + (\MU_{(n,m)})_{d+3,d+3}
    \Bigr] \notag \\[-7pt]
    & \hspace{2cm} + \Bigl[
    \smashoperator[r]{\sum_{(m,n) \in E}} \,
    (\MU_{(m,n)})_{d+4,d+4} + (\MU_{(m,n)})_{d+5,d+5}
    \Bigr].
\end{align*}
for $n \in V$.
ADMM for \eqref{eq:ADMM_Tik}
has the following explicit form.

\begin{theorem}
    \label{thm:solADMM_Tik}
    Let $\nu_n \coloneqq \lvert\{m \in V : (n,m) \in E \rvert +  \lvert\{m \in V : (m,n) \in E\}\rvert$
    for $n \in V$.
    For $\rho > 0$,
    \textup{ADMM} \eqref{eq:ADMM}
    applied to \eqref{eq:ADMM_Tik}
    reads as 
    \begin{align*}
        \Vx^{(k+1)}_n 
        &\coloneqq
        \proj_{\Aff_{d+1}} \bigl(\tfrac{1}{4\nu_n} \,\bigl(\FQ_{\Vx}^*(\MU^{(k)} - \MZ^{(k)})_n + \tfrac{1}{\rho} \, \Vy_n\bigr)\bigr)
        \quad \forall n \in V,
        \\
        \Vell^{(k+1)} 
        &\coloneqq
        \tfrac{1}{4\rho} \,
        \FQ_{\Vell}^*(\MU^{(k)} - \MZ^{(k)}), 
        \\
        \Vf^{(k+1)} 
        &\coloneqq
        \tfrac{1}{4} \,
        \bigl(\FQ_{\Vf}^*(\MU^{(k)} - \MZ^{(k)}) 
        + \tfrac{\lambda}{\rho} \, \Vones_M \bigr) 
        \\
        \Vv^{(k+1)}_n 
        &\coloneqq
        \tfrac{1}{2\nu_n} \,
        \bigl(\FQ_{\Vv}^*(\MU^{(k)} - \MZ^{(k)})_n 
        - \tfrac{1 + \nu_n\lambda}{2\rho}\bigr)
        \quad \forall n \in V,
        \\
        \MU_{(n,m)}^{(k+1)}
        &\coloneqq
        \proj_{\succeq \Mzero}\bigl(\FQ_{(n,m)}([\Vx, \Vell, \Vf, \Vv]^{(k+1)}) + \ME + \MZ_{(n,m)}^{(k)}\bigr) - \ME
        \quad\forall (n,m) \in E, 
        \\
        \MZ^{(k+1)}
        &\coloneqq
        \MZ^{(k)}
        + \FQ(\MU^{(k+1)}) 
        - \MU^{(k+1)},
    \end{align*}
    where $\proj_{\Aff_{d+1}}$ is the projection onto $\Aff_{d+1}$,
    $\proj_{\succeq \Mzero}$ onto the positive semi-definite matrices,
    and $\Vones_\bullet$ is the all ones vector.
\end{theorem}

The projections can be computed explicitly,
cf.\ \cite[Thm.~5.2]{BeBrSt24} and \cite[p.~399]{BV04}.
For $\Vxi \in \R^{d+1}$ 
and the eigenvalue decomposition $\MA = \MV \MSigma \MV^* \in \Msym_{d+5}$,
we have
\begin{subequations}
    \label{eq:projections}
    \begin{align}
        \proj_{\Aff_{d+1}}(\Vxi) 
        &= 
        (\xi_1, \dots, \xi_d, \max\{\xi_{d+1}, 1\})^*,
        \\
        \proj_{\succeq \Mzero} (\MA)
        &=
        \MV\max(\MSigma,0)\MV^*.
    \end{align}
\end{subequations}

\begin{proof}
    The minimizer of \eqref{eq:ADMM_I} can be easily computed 
    as in \cite[Thm.~5.1]{BeBrSt24}.
    The objective with $\FF$ chosen as $\FJ + \FI$ is quadratic
    and decouples all variables. 
    In order to compute the derivative,
    we can exploit
    \begin{align*}
        \FQ^*_{\Vx}(\FQ(\Vx, \Vell, \Vf, \Vv))_n 
        &= 
        4 \Vx_n,
        &
        \FQ^*_{\Vell}(\FQ(\Vx, \Vell, \Vf, \Vv))_{(n,m)} 
        &=
        4 \Vell_{(n,m)},
        \\
        \FQ^*_{\Vf}(\FQ(\Vx, \Vell, \Vf, \Vv))_{(n,m)} 
        &=
        4 \Vf_{(n,m)},
        &
        \FQ^*_{\Vv}(\FQ(\Vx, \Vell, \Vf, \Vv))_{n} 
        &= 
        2 \Vv_{n}.
    \end{align*}
    For $(\Vx_n)_{d+1}$,
    we have the constraint $(\Vx_n)_{d+1} \ge 1$,
    which results in the projection onto $\Aff_{d+1}$.
    In \eqref{eq:ADMM_II},
    $\MU_{(n,m)}$ are decoupled,
    and we have
    \begin{align*}
        \MU_{(n,m)}^{(k+1)}
        &=
        \argmin_{\MU \succeq - \ME} \, 
            \lVert\FQ_{(n,m)}([\Vx, \Vell, \Vf, \Vv]^{(k+1)}) - \MU + \MZ_{(n,m)}^{(k)}\rVert_{\tF}^2
        \\
        &= 
        \Bigl[\,\argmin_{\tilde \MU \coloneqq \MU + \ME \succeq \Mzero} \; 
        \lVert(\FQ_{(n,m)}([\Vx, \Vell, \Vf, \Vv]^{(k+1)}) + \ME + \MZ_{(n,m)}^{(k)}) - \tilde\MU\rVert_{\tF}^2 \Bigr] - \ME  \\
        &= 
        \proj_{\succeq \Mzero}(\MQ_{(n,m)}^{(k+1)} + \MZ_{(n,m)}^{(k)}) - \ME.
        \tag*\qed
    \end{align*}
\end{proof}

% -------------------------------------------------------------------------------------
\subsection{Solving the Relaxed TV denoising model}   \label{sec:sol_relaxed_tv}
% -------------------------------------------------------------------------------------

Similar to §~\ref{sec:sol_relaxed_tik},
we relay on the following splitting,
to solve \eqref{eq:conv_red_tik_Hd_TV}:
\begin{align}
    \label{eq:ADMM_TV}
    \argmin_{\Vx,\Vv,\Vu,\MU} \;
    \FK(\Vx, \Vv)
    + 
    \FI(\Vu)
    +
    \FH(\MU)
    \quad\text{s.t.}\quad
    \Vx = \Vu
    \;\;\text{and}\;\;
    \FV(\Vx, \Vv) = \MU \in \Msym_{d+3}^N.
\end{align}
Here,
$\FI$ is defined as in \eqref{eq:indicator_func}
and 
\begin{align*}
    \FH(\MU) \coloneqq \sum_{n \in V} \iota_{\hat {\mathcal C}}(\MU_{n}),
\end{align*}
where $\hat\ME \coloneqq \diag(\MI_{d+1}, 
\left[\begin{smallmatrix}
        0 & -1 \\
        -1 & 0
    \end{smallmatrix}\right])$
and $\hat{\mathcal C} \coloneqq \{\MA \in \Msym_{d+3} : \MA + \hat \ME \succeq \Mzero\}$.
The linear operator $\FV \coloneqq (\FV_n)_{n \in V}$
is defined as 
\begin{equation*}
    \FV_{n}(\Vx,\Vv) \coloneqq \MV_{n} - \hat\ME.
\end{equation*}

To apply \eqref{eq:ADMM},
we define $\lVert\MU\rVert^2 \coloneqq \sum_{n \in V} \lVert\MU_{n}\rVert_{\tF}^2$
and derive the adjoint of $\FV$
that%
---with respect to the arguments $\Vx$
and $\Vv$---%
is given by 
\begin{subequations}
    \label{eq:adj_V}
    \begin{align}       
        \bigl((\FV_{\Vx}^* (\MU))_{n}\bigr)_i
        & = 
        2 \, \bigl[ (\MU_{n})_{i,d+2} 
        + (-1)^{\delta_{i,d+1}} (\MU_{n})_{i,d+3} \bigr], 
        \label{eq:adj_x_V}
        \\
        (\FV_{\Vv}^* (\MU))_{n} 
        & = (\MU_{n})_{d+2,d+2} + (\MU_{n})_{d+3,d+3}
        \label{eq:adj_v_V}
    \end{align}
\end{subequations}
for $n \in V$ and $1 \leq i \leq d+1$.
ADMM for \eqref{eq:ADMM_TV} has the following form.

\begin{theorem}
    \label{thm:solADMM_TV}
    For $\rho > 0$,
    \textup{ADMM} \eqref{eq:ADMM} applied to \eqref{eq:ADMM_TV}
    reads as
    \begin{align*}
        \Vx^{(k+1)}
        &\coloneqq
        \prox\nolimits_{\nicefrac{\mu}{5\rho} \TV}
                \bigl(\tfrac{1}{5} [ \FV^*_{\Vx}(\MU^{(k)} - \MZ^{(k)}) 
                + \tfrac{1}{\rho} \, \Vy 
                + (\Vu^{(k)} - \Vz^{(k)})]\bigr),
        \\
        \Vv^{(k+1)}
        &\coloneqq
        \tfrac{1}{2\rho}
        \bigl(\rho\FV^{*}_{\Vv}(\MU^{(k)} - \MZ^{(k)}) 
        - \tfrac{1}{2} \Vones_N \bigr), 
        \\
        \MU^{(k+1)}_{n} 
        &\coloneqq
        \proj_{\succeq \Mzero}(\MV_n^{(k+1)} + \MZ_n^{(k)}) 
        - \hat\ME
        \quad \forall n \in V, 
        \\
        \Vu^{(k+1)}_n
        &\coloneqq 
        \proj_{\Aff_{d+1}}(\Vx^{((k+1)}_n + \Vz^{(k)}_n)
        \quad \forall n \in V,
        \\
        \MZ^{(k+1)}
        &\coloneqq
        \MZ^{(k)} + \FV([\Vx, \Vv]^{(k+1)}) - \MU^{(k+1)}, \\
        \Vz^{(k+1)} 
        &\coloneqq 
        \Vz^{(k)} + \Vx^{(k+1)} - \Vu^{(k+1)},
    \end{align*}
    where $\prox_{f} \coloneqq \argmin_{\Vy} \{f(\Vy) + \frac{1}{2}\lVert \,\cdot - \Vy\rVert^2\}$
    for convex, lower semi-continuous, proper $f$.
\end{theorem}

The projection onto $\Aff_{d+1}$
and the positive semi-definite matrices
is known from \eqref{eq:projections}.
Further,
we rely on an anisotropic TV regularization.
Therefore, 
we can apply the fast TV programs \cite{Con12v2, Con13v4} coordinatewise
to efficiently compute $\prox_{\nicefrac{\mu}{4\rho}\TV}$
for line and grid graphs, 
i.e.\ for signals and images,
cf.\ \cite[Rem.~4.1]{BeBr24}.

\begin{proof}
    In the proof,
    we denote by $\llbracket \MA \rrbracket_j$ 
    the first $d+1$ entries of the $j$-th column of $\MA \in \Msym_{d+3}$.
    Furthermore, 
    let $\Ve \coloneqq (1,\cdots,1,-1)^* \in \R^{d+1}$.
    Then,
    the adjoint in \eqref{eq:adj_x_V} may be written as
    \begin{align*}
        (\FV_{\Vx}^* (\MU))_{n}
        &= 2 \, \bigl( \llbracket \MU_n \rrbracket_{d+2} + \Ve \odot \llbracket \MU_n \rrbracket_{d+3} \bigr)
        \quad \forall n \in V,
    \end{align*}
    where $\odot$ denotes the Hadamard product.
    In the first ADMM step \eqref{eq:ADMM_I}
    for our splitting \eqref{eq:ADMM_TV},
    the minimization with respect to $\Vx$ and $\Vv$ decouples.
    Using the symmetry of $\MU^{(k)}$ and $\MZ^{(k)}$,
    exploiting $\tilde \Vx_n = \Ve \odot \Vx_n$,
    and combining the squared norms and inner products,
    we can compute the minimizer with respect to $\Vx$ by
    \begin{align*}
        & \argmin_{\Vx \in (\R^{d+1})^N}
        \Bigl\{\tfrac{1}{2}\sum_{n \in V} 
        (\Vv_n - 2\langle \Vx_n,\Vy_n\rangle)
        + \mu \TV(\Vx)
        + \tfrac{\rho}{2}\sum_{n \in V}
        \lVert \Vx_{n} - \Vu_n^{(k)} + \Vz_n^{(k)}\rVert_2^2 
        \\
        &\hspace{6.39cm}+ 
        \tfrac{\rho}{2}\sum_{n \in V}
        \rVert\FV_{n}(\Vx,\Vv) - \MU_n^{(k)} + \MZ_n^{(k)}\lVert_{\tF}^2 
        \Bigr\}
        \\
        &= 
        \argmin_{\Vx \in (\R^{d+1})^N}
        \Bigl\{-\smashoperator{\sum_{n \in V}} \langle \Vx_n,\Vy_n\rangle
        + \mu \TV(\Vx)
        + \rho \sum_{n \in V} 
        \Bigl[
        \tfrac{1}{2}
        \lVert\Vx_{n} - \Vu_n^{(k)} + \Vz_n^{(k)}\rVert_2^2
        \\
        &\hspace{1.2cm}
        + 
        \lVert\Vx_{n} - \llbracket\MU_{n}^{(k)}\rrbracket_{d+2} + \llbracket \MZ_{n}^{(k)} \rrbracket_{d+2}\rVert_2^2 
        + 
        \lVert\tilde\Vx_{n} - \llbracket \MU_{n}^{(k)} \rrbracket_{d+3} + \llbracket \MZ_{n}^{(k)} \rrbracket_{d+3} \rVert_2^2  
        \Bigr]\Bigr\}
        \\
        &= \argmin_{\Vx \in (\R^{d+1})^N}
        \Bigl\{-\smashoperator{\sum_{n \in V}} \langle \Vx_n,\Vy_n\rangle
        + \mu \TV(\Vx)
        \\[-10pt]
        &\hspace{4.11cm}
        + \tfrac{5\rho}{2}
        \sum_{n \in V} 
        \lVert \Vx_{n} 
        - \tfrac{1}{5}
        [\FV^*_{\Vx}(\MU - \MZ)_n + (\Vu^{(k)} - \Vz^{(k)})_n]\rVert_2^2 \Bigr\} \\
        & = \argmin_{\Vx \in (\R^{d+1})^N}
        \Bigl\{
        \tfrac{5\rho}{2}\sum_{n \in V} 
        \lVert\Vx_{n} - \tfrac{1}{5}[\FV^*_{\Vx}(\MU - \MZ)_{n} + \tfrac{\Vy_n}{\rho} + (\Vu_n - \Vz_n)]\rVert_2^2 
        + \mu \TV(\Vx)\Bigr\},
    \end{align*}
    Because of the differentiability of the objective with respect to $\Vv$,
    and since $\FV^*_{\Vv}(\FV(\Vx, \Vv)) = 2\Vv$,
    the minimizer is given by
    \begin{align*} 
        \Mzero_N
        &= 
        \tfrac{1}{2} \Vones_N 
        + \rho\FV_{\Vv}^*(\FV_{\Vv}(\Vx,\Vv) - \MU^{(k)} + \MZ^{(k)}), 
        \quad\text{and thus}
        \\
        \Vv^{(k+1)}
        &= 
        \tfrac{1}{2}
        [\FV^{*}_{\Vv}(\MU^{(k)} - \MZ^{(k)})
        - \tfrac{1}{2\rho} \Vones_N].
    \end{align*}
    The two subproblems with respect to $\MU$ and $\Vu$ 
    of the second ADMM step \eqref{eq:ADMM_II} regarding \eqref{eq:ADMM_TV}
    can be treated similarly to the proof of Thm.~\ref{thm:solADMM_Tik}
    yielding the projection onto the positive semi-definite matrices and $\Aff_{d+1}$.
    \qed
\end{proof}

% ------------------------------------------------------------------------------------------
\section{Numerical Experiments}         \label{sec:num-exp}
% ------------------------------------------------------------------------------------------

The derived algorithms in Thm.~\ref{thm:solADMM_Tik} and \ref{thm:solADMM_TV} are implemented\footnote{
The code is available at GitHub: \url{https://github.com/JJEWBresch/relaxed_tikhonov_regularization}.} 
in Python~3.11.4 using Numpy~1.25.0 and Scipy~1.11.1.
The experiments are performed on an off-the-shelf iMac 2020
with Apple M1 Chip (8‑Core CPU, 3.2~GHz) and 8~GB RAM.
The convergence of both algorithms 
is ensured by \cite[Cor.~28.3]{bauschke}.
For all experiments,
the methods are initialized by zeros for all variables.
The resulting implementations are very similar to the ones 
proposed for sphere- and SO(3)-valued data in
\cite{BeBrSt24,BeBr24}.

%\todo[inline]{Experiments in:
%-- synthetic data 
%-- Gaussian image processing
%-- Giga camera.}

\paragraph{Synthetic Image Processing.}
In this subsection,
we give a proof of the concept for both models%
---the Tikhonov-like 
and TV model---%
based on synthetic data.
For the ground truth $\Vx_{\text{true}}$
of the two experiments,
see~Fig.~\ref{fig:H1/H2-signals},
we interpolate points in $\R$ or $\R^2$
and apply the smooth mappings
\begin{subequations}   
    \label{eq:paraHd}
    \begin{align}
        \Vx_{\hyper_1} 
        &\colon 
        \R \to \hyper_1,
        &&
        r \mapsto (\sinh(r), \cosh(r))^*,
        \label{eq:paraH1}
        \\
        \Vx_{\hyper_2} 
        &\colon 
        \R^2 \to \hyper_2,
        &&
        (r,s) \mapsto (\sinh(r) \cos(s), \sinh(r) \sin(s), \cosh(r))^*.
        \label{eq:paraH2}
    \end{align}
\end{subequations}

In the first experiment,
where an one-dimensional signal of length 400,
i.e.\ a signal on a line graph,
is considered,
the noisy measurements $\Vy \in \hyper_1^{400}$ are generated employing
the tangential normal distribution \cite{laus} 
with standard deviation $\sigma \coloneqq 0.6$,
see Fig.~\ref{fig:H1/H2-signals} (top).
The $\hyper_1$-valued signals are here illustrated 
via the parameter $r$ in \eqref{eq:paraH1}.
Applying ADMM from Thm.~\ref{thm:solADMM_Tik} and \ref{thm:solADMM_TV}
for the Tikhonov-like and TV model respectively,
we observe convergence to the hyperbolic sheet.
In both cases,
the mean absolute error (MAE)
with respect to the Minkowski bilinearform $\Mbiform \equiv -1$
is $10^{-4}$ at the most.
The restored signals after 12 sec. and the used parameters 
are recorded in Fig.~\ref{fig:H1/H2-signals} (top).

\begin{figure}[t]
    \begin{minipage}{\linewidth}
    \includegraphics[width=\linewidth, clip=true, trim=110pt 5pt 110pt 15pt]{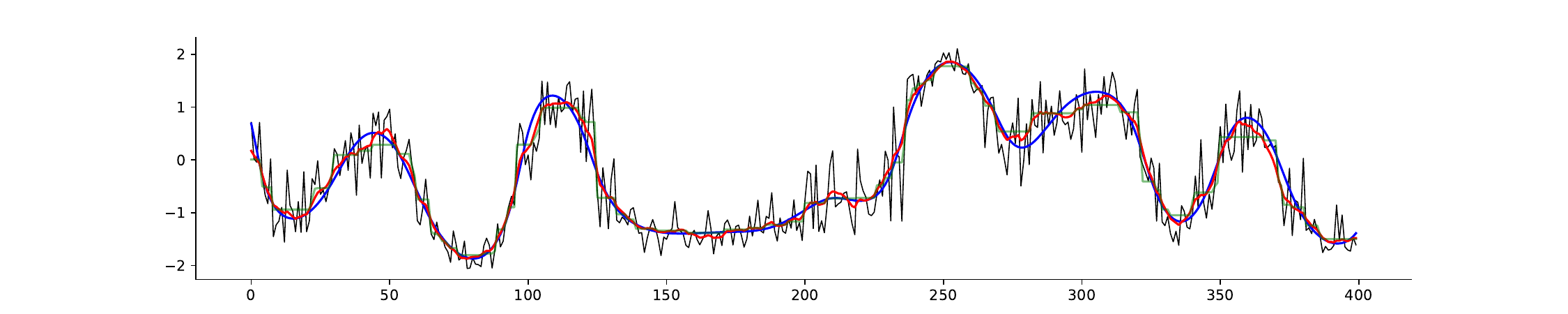}
    %{1-hyperbolic/1-hyperbolic_data_denoise_presentation_sig=0.6.pdf}
    \end{minipage}
    \begin{minipage}{0.54\linewidth}
    \caption{Restoration of smooth synthetic line signals (blue)
    from noisy measurements (black/gray).
    For the Tikhonov model (red),
    we choose $\rho = 10^{-1}$,
    $\lambda = 6$ (top) and
    $\lambda = 5$ (right).
    For the TV model (green),
    we choose
    $\rho = 1$,
    $\mu = 0.75$ (top),
    and $\mu = 0.1$ (right).
    Note that 
    the $\hyper_1$ signal (top) is visualized 
    using the parametrization in \eqref{eq:paraH1}.
    The $\hyper_2$ signal (right) is directly visualized
    on the hyperbolic sheet in $\Aff_3 \subset \R^3$.
    }
    \label{fig:H1/H2-signals}
    \end{minipage}
    \hspace{0.0cm}
    \begin{minipage}{0.44\textwidth}
    \includegraphics[width=\linewidth, clip=true, trim=150pt 250pt 120pt 120pt]{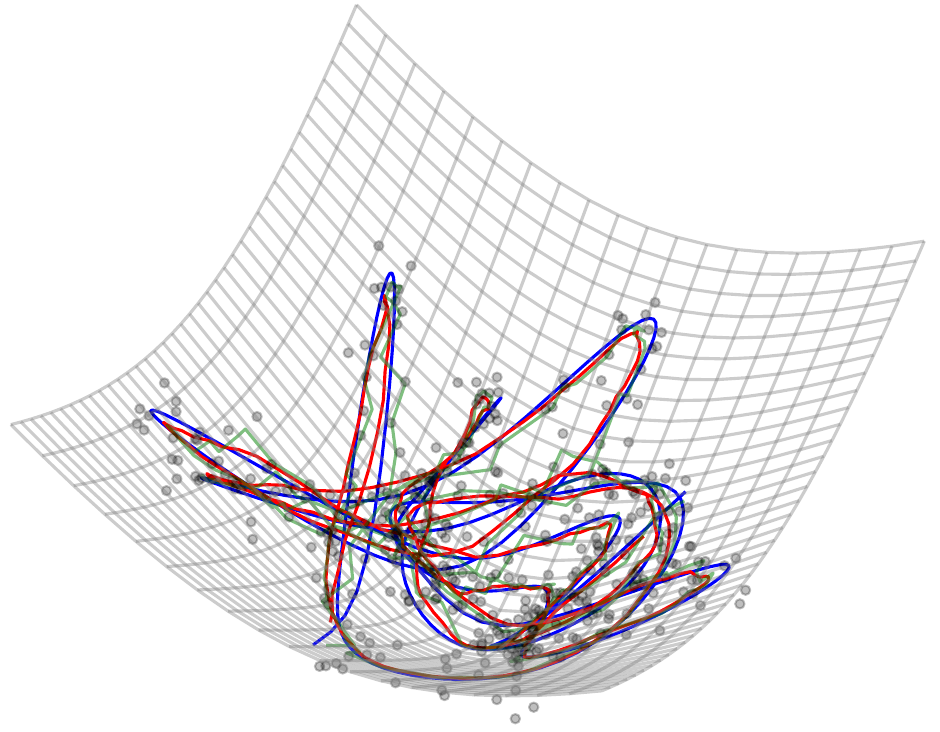}
    %{2-hyperbolic/2-hyperbolic_denoising_sig_0.3.pdf}
    \end{minipage}
\end{figure}

In the second experiment,
we consider a line signal of length 400
with values in $\hyper_2$.
The noisy measurements are generated via 
additive Gaussian noise,
i.e.\ $\Vy \sim \mathcal N(\Vx_{\text{true}}, \sigma^2 \MI)$
with $\sigma \coloneqq 0.3$.
Note that the measurements lie in $\R^{3}$ and not necessarily on $\hyper_2$.
The denoised signals using ADMM for both models
are shown in Fig.~\ref{fig:H1/H2-signals} (right).
After 80 sec.,
we recorded an MAE 
with respect to $\eta \equiv -1$
of $10^{-5}$ at the most.
Because of convergence back to the hyperbolic sheet, 
the numerical solutions are actual solutions of the non-convex model
\eqref{eq:orig_tik_Hd} and \eqref{eq:orig_tik_Hd_TV} respectively.

\begin{figure}[t]
\begin{minipage}{0.63\linewidth}
    \includegraphics[width=\linewidth, clip=true, trim=15pt 25pt 15pt 38pt]{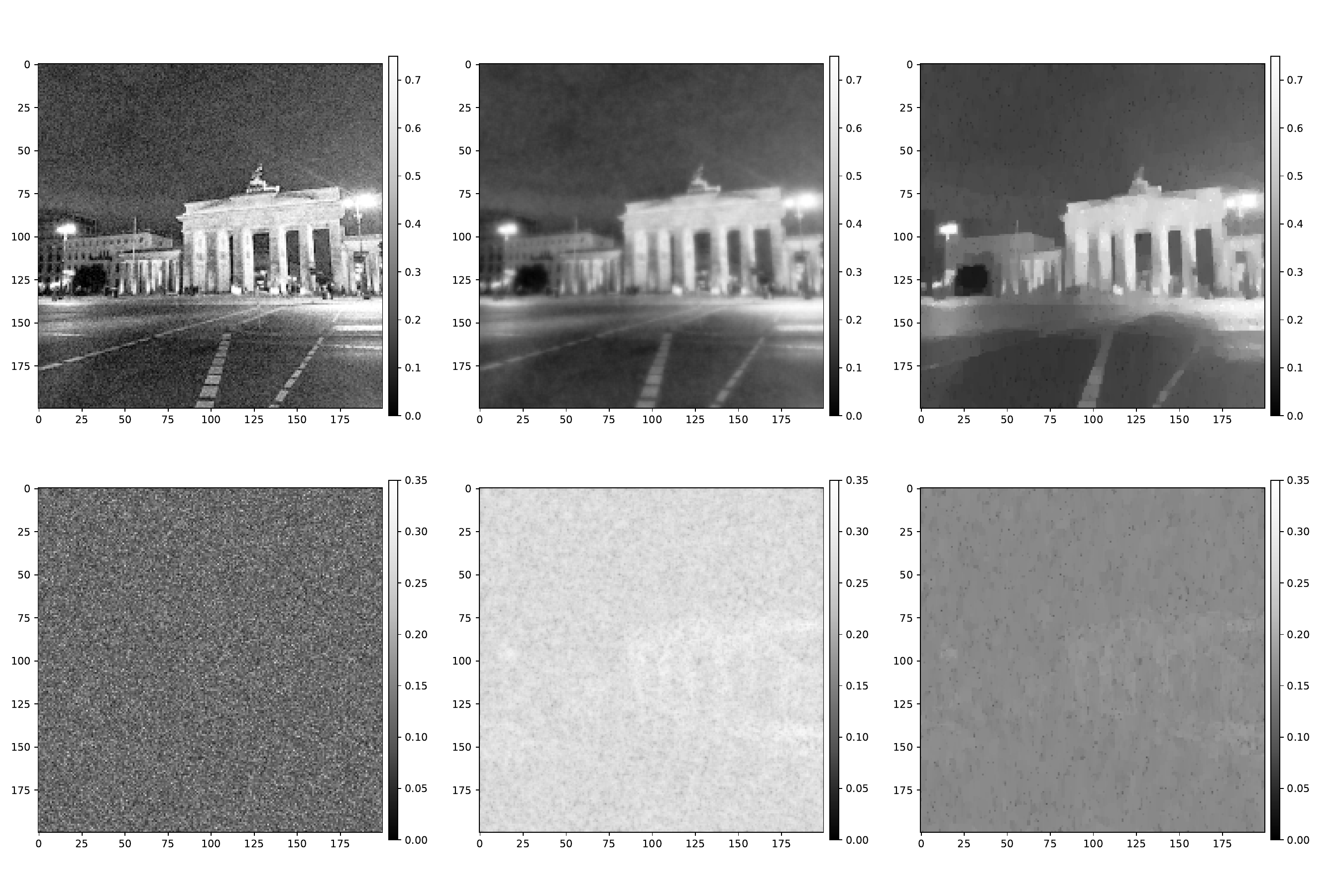}
\end{minipage}
\begin{minipage}{0.35\linewidth}
    \caption{
    Empirical mean (top) and standard deviation (bottom)
    for the Brandenburg Gate image.
    Here,
    $K = 20$ shots are randomly sampled 
    by adding white noise 
    with $\sigma = 0.15$.
    The empirical quantities are shown on the left,
    the Tikhonov denoised images (Thm.~\ref{thm:solADMM_Tik}, $\rho = 10$, $\lambda = 4$) in the middle,
    and the TV denoised images (Thm.~\ref{thm:solADMM_TV}, $\rho = 1$, $\mu=0.6$) on the right.}
    \label{fig:quant_H2}
\end{minipage}
\end{figure}

\paragraph{Gaussian Image Processing.}
The idea behind Gaussian image processing is
that,
instead of a single image,
we measure a series of $K$ noisy instances,
where we assume that each pixel $x_{i,j}^{(k)}$
follows a Gaussian distribution.
The empirical mean and variance may be estimated 
by the maximum likelihood estimators
\begin{equation}        
    \label{eq:estMeanSigMaxLike}
    \hat \mu_{ij} \coloneqq \frac{1}{K}\sum_{k = 1}^K x_{i,j}^{(k)},
    \qquad
    \hat \sigma_{ij}^2 \coloneqq \frac{1}{K}\sum_{k = 1}^K (x_{i,j}^{(k)} - \hat\mu_{ij})^2.
\end{equation}
The set of Gaussians $\mathcal N(\mu,\sigma)$ parameterized by $(\mu, \sigma) \in \R \times \R_{>0}$
and equipped with the so-called Fisher metric
is isometric to the hyperbolic space $\hyper_2$.
More precisely,
the isometry is given by 
$\pi_3\circ\pi_2\circ\pi_1\colon \mathcal N \to \hyper_2$
with
\begin{subequations}
    \label{eq:isometry}
    \begin{align}
        &
        \pi_1 \colon \mathcal N \to \Phalf_2,
        &(\mu,\sigma) 
        &\mapsto 
        \tfrac{1}{\sqrt{2}}
        (\mu, \sqrt{2}\sigma)^*,
        \\
        &
        \pi_2 \colon \Phalf_2 \to \Pdisc_2,
        &
        \Vy 
        &\mapsto 
        \tfrac{1}{y_1^2 + (y_2+1)^2} 
        (2 y_1, \lVert\Vy\rVert_2^2 - 1)^*,
        \\
        & \pi_3 \colon \Pdisc_2 \to \hyper_2,
        &
        \Vy 
        &\mapsto
        \tfrac{1}{1-\lVert\Vy\rVert_2^2}
        (2 y_1, 2 y_2, 1 + \lVert\Vy\rVert_2^2)^*, 
    \end{align}
\end{subequations}
where $\Pdisc_2 \coloneqq \{\Vx \in \R^2 : \lVert\Vx\rVert_2<1\}$
is the Pioncaré disc,
and the Pioncaré half-space
is given by $\Phalf_2 \coloneqq \{\Vx \in \R^2 : x_2 > 0\}$, 
see \cite{BerPerSte16,L97}.

In the first experiment,
as proof of concept,
we consider a gray-valued image of the Brandenburg Gate
and generate a series of 20 images by
adding Gaussian white noise 
with constant standard deviation.
Calculating the empirical mean and variance pixelwise,
and using the isometry in \eqref{eq:isometry},
we obtain an $\hyper_2$-valued image.
Afterwards we denoise mean and variance simultaneously 
using the algorithms from Thm.~\ref{thm:solADMM_Tik} and \ref{thm:solADMM_TV}.
The results are recorded in Fig.~\ref{fig:quant_H2}.
Analyzing the qualitative outcomes,
we conclude 
that both models are adequate ansätze.
The Tikhonov-like model, however, overestimates the standard deviation.

In the second experiment,
we consider a real-world application of Gaussian image processing,
where 20 highly noisy images of a retina are given. 
The experiment originates from \cite{AV14},
and the goal is again to denoise mean and variance simultaneously.
The qualitative results of the Tikhonov-like and TV denoising
are shown in Fig.~\ref{fig:2D-H2-image}.
The Tikhonov-like model 
again overestimates the standard deviation.
During the denoising,
we observe numerical convergence to the hyperbolic sheet.
More precisely,
we reach an MAE of at most $10^{-4}$
with respect to $\eta \equiv -1$
in 15 min.\ for the Tikhonov-like model
and in 3 min.\ for the TV model.
The numerical solutions are thus
solutions of the non-convex models
\eqref{eq:orig_tik_Hd} and \eqref{eq:orig_tik_Hd_TV}, too.
For the TV proximity mapping in Thm.~\ref{thm:solADMM_TV},
we use the TV program (ADRA) \cite{Con12v2} 
with $\gamma = 1$ as subprocedure.
A quantitative comparison 
with the Proximal Douglas--Rachford Algorithm (PDRA) \cite{BerPerSte16}
for more generally manifold-valued data
is given in Tab.~\ref{tab:com_H2_images}.
Our relaxed Euclidean models here yield comparable results
with respect to the SNR (signal-to-noise ratio)
and the computation time.

\subsubsection{Acknowledgements} 
We would like to thank R. Bergmann, J. Persch and G. Steidl
for providing the data of the retina experiment,
and G. Steidl for drawing the attention to this experiments.
Originally,
the data was provided by J. Angulo.

\begin{figure}[t]
\begin{minipage}{0.231\linewidth}
\includegraphics[width=\linewidth, clip=true, trim=1pt 0pt 0pt 2pt]{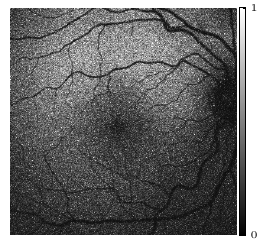}
\includegraphics[width=\linewidth, clip=true, trim=0pt 0pt 2pt -14pt]{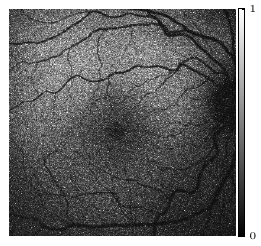}
\end{minipage}
\raisebox{-2.8cm}{\rule{1pt}{5.75cm}}
\begin{minipage}{0.84\linewidth}
\includegraphics[width=0.915\textwidth, clip=true, trim=5pt 35pt 25pt 50pt]{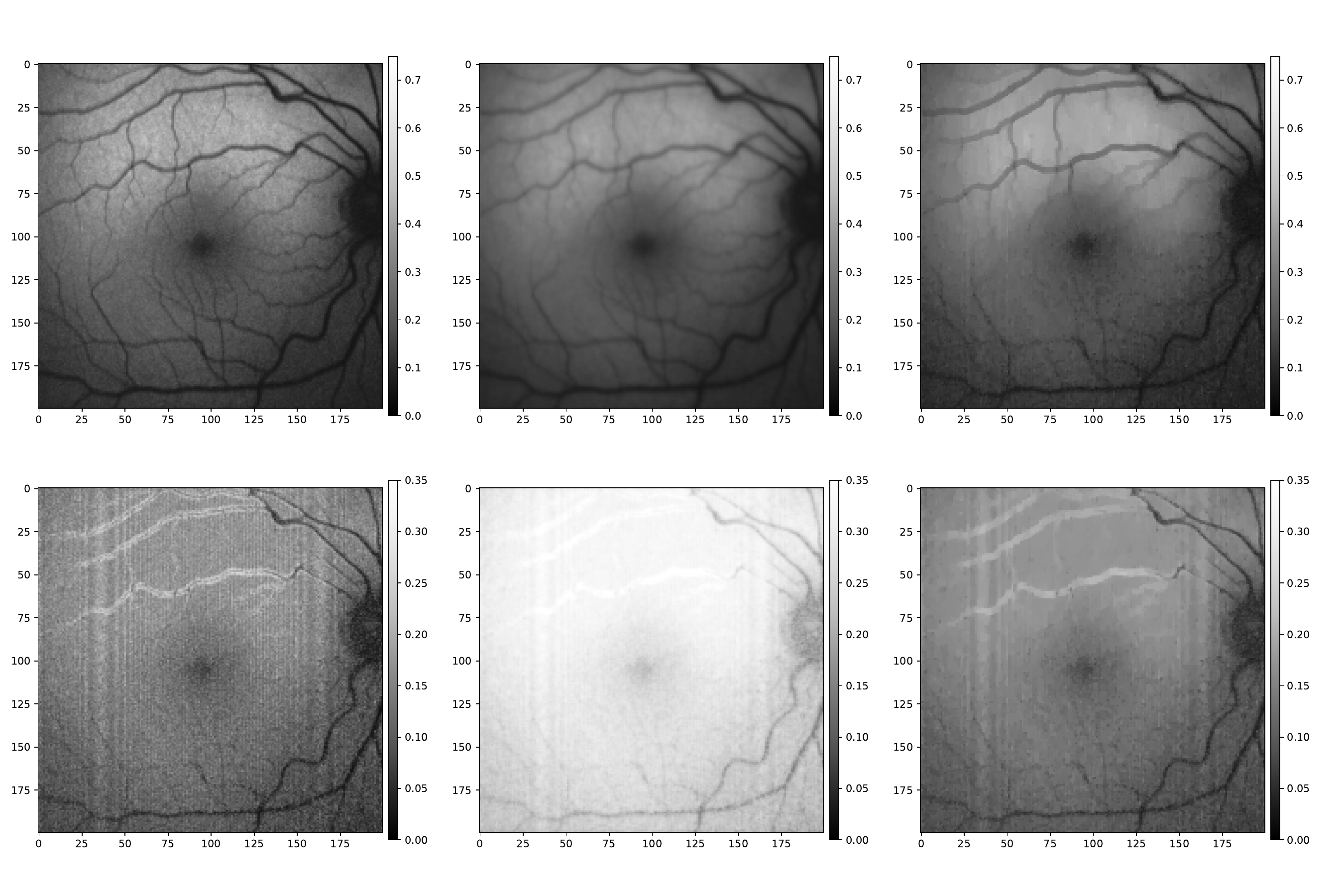}
\end{minipage}
\vspace{-0.45cm}
\caption{
Clearing the estimated mean 
and standard deviation of $K = 20$ retina scans.
\textit{Leftmost:} 
the 1st (top) 
and 20th (bottom) image 
of the retina scans.
\emph{From left to right:}
the empirical mean (top)
and standard deviation (bottom);
Tikhonov denoised images 
(Thm.~\ref{thm:solADMM_Tik}, $\lambda = 1.5$, $\rho = 10$);
TV denoised images (Thm.~\ref{thm:solADMM_TV}, $\mu = 0.15$, $\rho = 1$).
%For a detailed comparison concerning the TV model
%we refer to Tab.~\ref{tab:com_H2_images}.
}
\label{fig:2D-H2-image}
\end{figure}

\begin{table}[t]
    \footnotesize
    \resizebox{\textwidth}{!}{
    \begin{tabular}{c@{\quad}c@{\quad}c@{\quad}c@{\quad}c@{\quad}c | c@{\quad}c@{\quad}c@{\quad}c}
        \toprule
        \multicolumn{6}{c}{ADMM from Thm.~\ref{thm:solADMM_TV} ($\rho = 1$)} & \multicolumn{4}{|c}{PDRA ($\eta = 0.5, \lambda = 0.9$)} \\
        SNR($\hat\mu$) & SNR($\hat\sigma$) & $\mu$ & $\FK$ & MEA $\eta \equiv -1$ & time (sec.) & SNR($\hat\mu$) & SNR($\hat\sigma$) & $\alpha$ & time (sec.) \\
        \toprule
        $6.121$ & $12.211$ & $0.05$ & $2.61 \cdot 10^5$ & $4.73 \cdot 10^{-4}$ & 160 
        & 
        $5.963$ & $12.400$ & $0.1$ & 120
        \\
        $6.001$ & $12.369$ & $0.10$ & $2.54 \cdot 10^5$ & $4.78 \cdot 10^{-4}$ & 180 
        &
        $5.911$ & $12.618$ & $0.3$ & 140 
        \\
        $\mathbf{5.986}$ & $\mathbf{12.301}$ & $0.15$ & $1.81 \cdot 10^5$ & $4.47 \cdot 10^{-4}$ & 190 
        & 
        $5.874$ & $12.748$ & $0.5$ & 280
        \\
        $5.966$ & $12.455$ & $0.20$ & $1.19 \cdot 10^5$ & $4.67 \cdot 10^{-4}$ & 220 
        & 
        $5.843$ & $12.849$ & $0.7$ & 430\\
        \bottomrule
    \end{tabular}
    }
    \vspace{0.05cm}
    \caption{
    Comparison of the TV denoising (Thm.~\ref{thm:solADMM_TV})
    and PDRA (Douglas--Rachford) \cite{BerPerSte16}
    regarding the SNR for the restored mean 
    and standard deviation.
    The marked instance ($\mu = 0.15$) is visualized in Fig.~\ref{fig:2D-H2-image}.}
    \label{tab:com_H2_images}
    \vspace{-0.8cm}
\end{table}

\vspace{-5pt}
\bibliographystyle{splncs04}
\bibliography{literatur}

\end{document}